\documentclass[a4paper,11pt]{article}

  \renewcommand\appendix{\par
  \setcounter{section}{0}
  \setcounter{subsection}{0}
  \setcounter{figure}{0}
  \setcounter{table}{0}
  \renewcommand\thesection{ Appendix \Alph{section}}
  \renewcommand\thefigure{\Alph{section}\arabic{figure}}
  \renewcommand\thetable{\Alph{section}\arabic{table}}
}

 \usepackage{amsmath, amsthm, amssymb}
\usepackage{amssymb}
\usepackage{graphicx}
\usepackage{algorithmic}
\usepackage{algorithm}
\usepackage{multicol}
\usepackage{supertabular}
\usepackage[margin=1.2in]{geometry}
\usepackage[T1]{fontenc}
\usepackage{pgfplots}

\pgfplotsset{compat=newest}

\usepackage{setspace}

\theoremstyle{remark}
\newtheorem{claime}{Claim}

\begin{document}


\title{A New Upper Bound on Total Domination Number of Bipartite Graphs}


%
%
\author{SAIEED AKBARI \and  POOYAN EHSANI \and SAHAR QAJAR \and ALI SHAMELI \and HADI YAMI}






%


%

\maketitle

\begin{abstract}
Let $ G $ be a graph. A subset $S \subseteq V(G) $ is called a total dominating set if every vertex of $G$ is adjacent to at least one vertex of $S$. The total domination number, $\gamma_{t}$($G$), is the minimum cardinality of a total dominating set of $G$. In this paper using a greedy algorithm we provide an upper bound for $\gamma_{t}$($G$), whenever $G$ is a bipartite graph and $\delta(G)$ $\geq$ $k$. More precisely, we show that if $k$ > 1 is a natural number, then for every bipartite graph $G$ of order $n$ and $\delta(G) \ge k$, $ $$\gamma_{t}$($G$) $\leq$ $n(1- \frac{k!}{\prod_{i=0}^{k-1}(\frac{k}{k-1}+i)}).$ 
\end{abstract}
\section{Introduction}
Let $G$ be a graph. In this paper all graphs are simple with no multiple edges and loops. A $dominating$ $set$ for a graph $G = (V, E)$ is a subset $D$ of $V$ such that every vertex in $V \setminus D$ is adjacent to at least one vertex of $D$.  The $domination$ $number$ $\gamma(G)$ is the number of vertices in a smallest dominating set of $G$.

A $total$ $dominating$ $set$ is a set of vertices such that all vertices in the graph (including the vertices in the total dominating set themselves) have a  neighbor in the total dominating set. The $total$ $domination$ $number$ $\gamma_{t}(G)$, is the number of vertices in a smallest total dominating set of $G$. We denote the minimum degree of $G$ by $\delta(G)$.

Domination problems came from chess. In the 1850s, several chess players were interested in the
minimum number of queens such that every square on the chess board either contains
a queen or is attacked by a queen. Domination has several applications. One of them is
facility location problems, where the number of facilities (e.g.,
hospitals, fire stations, ...) is fixed and one attempts to minimize the distance that a
person needs to travel to get to the closest facility. Another application is finding sets of representatives, in monitoring communication or
electrical networks, and in land surveying (e.g., minimizing the number of places a
surveyor must stand in order to take height measurements for an entire region).
In \cite{FundDom}, the authors cite more than 75 variations of domination.

Total domination in graphs was first introduced by Cockayne, Dawas, and Hedetniemi \cite{origintotal,Henning} and is now very popular in graph theory. There are two outstanding domination books by Haynes, Hedetniemi, and Slater \cite{FundDom, Dom, Henning} who have gathered and unified all results regarding this subject through 1200 domination papers at that time.
  
 The following interesting result is due to Alon, see \cite{Alon}.
\theoremstyle{plain} \newtheorem{theo}{Theorem}
\begin{theo}{\rm \cite{Alon}}
Let $G = (V, E)$ be a graph of order n, with  $\delta(G) > 1$. Then $G$ has a dominating set of size at most $n(\frac{1+\ln (\delta(G)+1)}{\delta(G)+1})$.
\end{theo}

In 2009 Henning found a similar result for total domination and he showed that his result was optimal for sufficiently large $\delta$ \cite{Henning}.

\begin{theo}{\rm \cite[p.4]{Henning}}\label{MHen}
Let $G = (V, E)$ be a graph of order n, with  $\delta(G) > 1$. Then $G$ has a total dominating set of size at most $n(\frac{1+\ln \delta(G)}{\delta(G)})$ vertices.
\end{theo}

Computing the exact value of total domination number of bipartite graphs is NP-complete \cite{Pfaff}.
 In this paper we provide a new upper bound which improves Theorem \ref{MHen} for bipartite graphs. First we obtain an upper bound for the total domination number of regular bipartite graphs and then we prove that our upper bound holds for all bipartite graphs. Given a $k$-regular bipartite graph $G$ of order $n$, with vertex set $V$ and edge set $E$, we would like to choose the minimum possible number of vertices in $V$, so that each vertex has at least one chosen neighbor. Suppose that $G$ has two parts $X$ and $Y$. We find an upper bound for the number of chosen vertices from one part, and due to symmetry we obtain an upper bound for the number of chosen vertices from $G$.
So we want to choose the minimum number of vertices from $X$ so that each vertex in $Y$ has a selected neighbor in $X$. 
To choose the vertices from $X$ so that all vertices from $Y$ have selected neighbors, we use a greedy algorithm as follows.

While there exists an uncovered vertex in $Y$(If a vertex $u \in Y$ has a chosen neighbor in $X$, we say $u$ is $covered$), we choose a vertex of $X$ which has the maximum number of uncovered vertices of $Y$, adjacent to itself. Eventually after calculations we obtain the following inequality: $$\gamma_{t}(G) \le n(1- \frac{k!}{\prod_{i=0}^{k-1}(\frac{k}{k-1}+i)}),$$ which is an improvement of Theorem \ref{MHen}.

\section{Definitions and Results}
\subsection{Definitions}
To achieve the aforementioned upper bound, we need to consider the problem in mathematical terms. Therefore we need to  define the following functions which will be used later in this paper.

We call $f: \mathbb{Z} \times \mathbb{Z} \rightarrow \mathbb{Z}^{\ge \,0}$ a {\it good function of order $k$}, if and only if for given $M_{0}, N_{0}$ we have $ f(M_{i}, N_{i}) = f(M_{i+1}, N_{i+1}) + 1 $, and $ M_{i+1} = M_{i} - k \lceil\frac{M_{i}}{N_{i}}\rceil $, $ N_{i+1} = N_{i} - 1 $ in which $1 < k \le N_{0}$ is a constant integer, $M_{0}$ and $N_{0}$ are natural numbers, and $k|M_{0}$. Also if $x \le 0$ or $y \le 0$, we define $f(x, y) = 0$.

We call $f:\mathbb{Z}\times \mathbb{Z} \rightarrow \mathbb{Z}^{\ge \,0} $ a {\it nice function of order $k$}, if and only if for given $M_{0}, N_{0}$ we have $ f(M_{i}, N_{i}) = f(M_{i+1}, N_{i+1}) + 1$ and $ M_{i+1} = M_{i} - k (\lceil\frac{M_{i}}{N_{i}}\rceil + x_{i}) $, $ N_{i+1} = N_{i} - 1 $ in which $1 < k \le N_{0}$ is a constant integer, $M_{0}$ and $N_{0}$ are natural numbers, all of $x_{i}$ are non-negative constant integers, and $k|M_{0}$. Also if $x \le 0$ or $y \le 0$ we define $f(x, y) = 0$. Now, we have following fact:

$Fact\, 1.$ For any good function of order $k$ like $f$, it is obvious that for any natural number $i$, $M_{i}$ is divisible by $k$.

\subsection{Our Results}
For any bipartite graph $G=(V, E)$, we will use Algorithm \ref{alg1} to obtain a total dominating set $S \in V$.

\begin{algorithm}                      
\caption{Finding a Total Domination for a Bipartite Graph}          
\label{alg1}                           
\begin{algorithmic}[1]                    
\STATE \textit{Input}:
     \STATE \hspace{\algorithmicindent}  $G=(V, E)$: a bipartite graph of order $n$ and $\delta(G) \ge k$ with parts $X$ and $Y$

\STATE \textit{Output}:
     \STATE \hspace{\algorithmicindent} \textit{$S \subseteq V$}: a total dominating set for graph $G$
     \STATE $G'=G$
     \STATE $U = \emptyset$
     \STATE $U'=\emptyset$
     \WHILE {\textit{$N(U)\ne Y$}}
        \STATE choose a vertex $u \in X$ from $G$ with maximum degree
        \STATE add $u$ to $U$
        \STATE remove all edges incident with $N(u)$ from $G$ 
     \ENDWHILE
      
     \WHILE {\textit{$N(U')\ne X$}}
        \STATE choose a vertex $u \in Y$ from $G'$ with maximum degree
        \STATE add $u$ to $U'$
        \STATE remove all edges incident with $N(u)$ from $G'$ 
     \ENDWHILE
     	
     \STATE $S=U\cup U'$
\end{algorithmic}
\end{algorithm}

For any bipartite graph $G$ of order $n$ and $\delta(G)\le k$, if Algorithm \ref{alg1} yields the set $S$, then we have:
$$\gamma_{t}(G) \le |S| \le n(1- \frac{k!}{\prod_{i=0}^{k-1}(\frac{k}{k-1}+i)}).$$

In the next section, we will prove this result.

\section{Analysis of Algorithm \ref{alg1}}

We defined good function and nice function in the previous section. First we will show some of their useful properties and then we will use them in proving our theorems.

\newtheorem{lemm}{Lemma} 

\begin{lemm}\label{monoton}
Let $n,x,y$ be non-negative integers. Then for every good function $f$ of order $k$, if $x \ge y$, then $f(x, n)\ge f(y, n)$.\end{lemm}

\theoremstyle{proof}
\begin{proof}We prove this lemma by induction on $x$. By definition of good function of order $k$, we know that both $x$ and $y$ must be multiples of $k$.  For any $n$, it is clear that $f(0, n) \ge f(0, n)$. Now, suppose that for any $n, x \ge y$, if $x < C$($C$ is a multiple of $k$), then $f(x, n) \ge f(y, n)$. Then we have to show that for any $n$, $x \ge y$, if $x = C$, then $f(x, n) \ge f(y, n)$. We have: 	$$f(x, n)=f(x-\lceil\frac{x}{n}\rceil k, n-1)+1,$$ $$f(y, n)=f(y-\lceil\frac{y}{n}\rceil k, n-1)+1.$$ To complete the proof we should show that the following inequality cannot hold:  $$x-\lceil\frac{x}{n}\rceil k < y-\lceil\frac{y}{n}\rceil k.$$  After simplifying we have: 	$$\frac{x}{k}-\lceil\frac{x}{n}\rceil < \frac{y}{k}-\lceil\frac{y}{n}\rceil.$$ Since both $x$ and $y$ are multiple of $k$ and $x \ge y$, there exists some non-negative integer $a$ such that $x=y+a k$. So we have: 	$$\frac{(y+a k)}{k}-\lceil\frac{y+a k}{n}\rceil<\frac{y}{k}-\lceil\frac{y}{n}\rceil.$$ So, we have:	$$a<\lceil\frac{y+a k}{n}\rceil-\lceil\frac{y}{n}\rceil. \qquad (1.1)$$ Also it is obvious that: $$\lceil\frac{y+a k}{n}\rceil \le \lceil \frac{y}{n} \rceil + \lceil\frac{a k}{n}\rceil. \qquad (1.2)$$
(1.1) and (1.2) imply that $a<\lceil\frac{a k}{n}\rceil$, which contradicts $k \le n$. 
\end{proof}

\begin{lemm} \label{decrease}	
For every good function $f$ of order $k$, and each $i > 0$ the following holds:

$$M_{i-1}-M_{i} \ge M_{i}-M_{i+1}.$$ \end{lemm}

\begin{proof}By contradiction assume that there exists some $i$ , $M_{i-1}-M_{i} < M_{i} - M_{i+1}$. Then we have:
$$ M_{i-1} - M_{i} < (M_{i-1} - k\lceil\frac{M_{i-1}}{N_{i-1}}\rceil) - (M_{i} - k\lceil\frac{M_{i}}{N_{i}}\rceil).$$
So,
$$\frac{M_{i-1}}{N_{i-1}} < \frac{M_{i}}{N_{i}}.$$
Therefore,
$$M_{i-1}N_{i} < M_{i}N_{i-1}.$$
This implies that, 
$$M_{i-1}N_{i-1} - M_{i-1} < N_{i-1}M_{i-1} - kN_{i-1}\lceil\frac{M_{i-1}}{N_{i-1}}\rceil.$$
So,
$$kN_{i-1}\lceil\frac{M_{i-1}}{N_{i-1}}\rceil < M_{i-1},$$
which is obviously a contradiction because $k > 1.$
\end{proof} 

\begin{lemm}\label{mainlemma}
Suppose that f is a good function of order $k$. Then for any natural number $n \ge k$, the following holds:
$$\frac{f(kn,\, n)}{n} \le 1- \frac{k!}{\prod_{i=0}^{k-1}(\frac{k}{k-1}+i)}.$$
\end{lemm}
\begin{proof}

In this case $M_{0} = kn$ and $N_{0} = n$. Assume that for any positive integer $a$, $T_{a}=\{x > 0$ $|$ $M_{x-1} - M_{x} = ak\}$ (it is clear from Fact 1 that $M_{x-1}-M_{x}$ is a multiple of $k$ and from Lemma \ref{decrease} for every $x \ge 0$,  $M_{x-1} - M_{x} \le M_{0} - M_{1}$). We define $t_{a}=\vert T_{a} \vert$. It follows immediately that, $$\frac{f(kn, \,n)}{n} = \frac{\sum_{i=1}^{k} t_{i}}{n}.$$

 According to Lemma \ref {decrease}, for any $a$, the elements of $T_{a}$ are $t_{a}$ consecutive integers. Now, consider the smallest $i$ for which $M_{i-1}-M_{i}=ak$. We have $M_{i - 1} - M_{i} = \cdots = M_{i-2+t_{a}} - M_{i-1+t_{a}} = ak$ and $M_{i-1+t_{a}} - M_{i+t_{a}} < ak$.
By the definition $\lceil\frac{M_{i-1+t_{a}}}{N_{i-1+t_{a}}}\rceil \le a - 1$.
 
Suppose that $N_{i-1} = bn$ (where $0 \le b \le 1$). By the definition $M_{i-1} - M_{i} = k\lceil\frac{M_{i-1}}{N_{i-1}}\rceil$, and so $a = \lceil\frac{M_{i-1}}{N_{i-1}}\rceil = \lceil\frac{M_{i-1}}{bn}\rceil$. Therefore we have $M_{i} \le abn$. According to Lemma \ref{monoton} the upper bound of  $f(M_{i-1}, N_{i-1})$ is achieved when $M_{i-1}$ is the biggest number and we have $M_{i-1} \le  abn$ for which $M_{i-1}$ is multiple of $k$. Now, suppose that $M_{i-1} = abn - \epsilon$. We have:

 $$\frac{M_{i-1} -t_{a} ak}{N_{i-1} -t_{a}} = \frac{M_{i-1+t_{a}}}{N_{i-1+t_{a}}}  \le a-1.$$

So for the upper bound of $f(M_{i-1}, N_{i-1})$ we have: 

$$ \frac{abn - \epsilon - t_{a} ak}{bn - t_{a}} \le a-1. \quad (3.1)$$

It follows from (3.1) that $\frac{t_{a}}{n} \ge \frac {b - \epsilon n}{ak \,-\, a\, - \,1} $. By Lemma \ref{monoton}, if $t_{a}$ decreases, then $f(abn - \epsilon, N_{i-1})$ becomes larger. So if we suppose that for each $a$, $t_{a}$ has the smallest possible value it can have, according to Lemma 1, $f(kn, \,n)$ becomes as large as possible. To find an upper bound for $f(kn, \,n)$ we suppose $\frac{t_{a}}{n}=\frac {b - \epsilon n}{ak - a - 1}$.

Since $\frac{f(kn, \,n)}{n} = \frac{\sum_{i=1}^{k} t_{i}}{n}$, to find an upper bound for $\frac{f(kn,\,n)}{n}$, we can calculate $\frac{\sum_{i=1}^{k}{t_{i}}}{n}$. For this purpose we define a new recursive function:

$$  g'(a, b) = g'(a-1, \,b -\frac{t_{a}}{n}) + \frac{t_{a}}{n}.\quad (3.2)$$ In this function $a$ should be a non-negative integer and for every $x$, define $g'(0, x)=0.$ 

Now, we rewrite the function (3.2) as follows:

$$ g'(a, b) = g'(a-1,\, b - \frac{b - \epsilon n}{ak-a+1}) + \frac {b - \epsilon n}{ak-a+1}.\quad (3.3)$$
We define the recursive function $g$ as follows:
$$ g(a, b) = g(a-1,\, b - \frac{b}{ak-a+1}) + \frac {b}{ak-a+1}.\quad (3.4)$$
 In this function $a$ should be a non-negative integer, $k > 1$ is a constant natural number and for every $x$, we define $g(0, x)=0.$ 

\begin{claime} 
In function (3.4), we have $g(a,\, b) = bg(a,\, 1).$
\end{claime}

\begin{proof}
We apply induction on $a$. It is clear that $g(1,\, b) = bg(1,\, 1)$ by the definition of function $g$.
Suppose that for each $ a < r $, $g(a,\, b) = bg(a,\, 1)$. Now, for $a = r$ we have:
$$g(r, b) = g(r - 1,\, b - \frac{b}{rk - r + 1} ) + \frac{b}{rk - r + 1},$$ 
$$g(r, 1) = g(r - 1,\, 1  - \frac{1}{rk - r + 1} ) + \frac{1}{rk - r + 1}.$$ 
By induction we find that $g(r-1, \,b - \frac{b}{rk - r +1} ) = bg(r - 1,\, 1 - \frac{1}{rk - r +1})$. Hence $g(r,\, b) = bg(r,\, 1)$ and therefore the proof is complete.
\end{proof}

\theoremstyle{remark}
\newtheorem{claimf}{Claim}
\begin{claime} 
In function (3.4), for any non-negative number $c \le b$, we have $g(a, b) = g(a, b - c) + g(a, c)$.
\end{claime}

\begin{proof}
According to Claim 1, we have $g(a, b-c) + g(a,c) = (b-c)g(a,1) + cg(a,1)$ which is equal to $bg(a,1)$ and is obviously equal to $g(a,b)$.
\end{proof}

\theoremstyle{remark}
\newtheorem{claimg}{Claim}
\begin{claime} 
In function (3.4), we have $g(a, b) \le b$.
\end{claime}

\begin{proof} 
According to Claim 1, we have $g(a, b) = bg(a,1)$ and by definition of $g$, $g(a, 1) < 1$.
\end{proof}

\theoremstyle{remark}
\newtheorem{claimh}{Claim}
\begin{claime} 
In functions (3.3) and (3.4), we have $g'(a, b) \le g(a,b)$.
\end{claime}

\begin{proof}
We apply induction on $a$. It is clear that $g'(1, b) \le g(1,b)$. Suppose for each $a < r$ we have $g'(a, b) \le g(a,b)$. Now, for $a = r$ using Claims 2 and 3 we have: 
$$g(a, b) = g(a-1,\, b - \frac{b}{ak-a+1}) + \frac {b}{ak-a+1}$$ 
$$\ge g(a-1,\, b - \frac{b}{ak-a+1}) + \frac {b}{ak-a+1} + g(a-1,\, \frac{\epsilon n}{ak-a+1}) - \frac{\epsilon n}{ak-a+1}$$ 
$$= g(a-1,\, b - \frac{b - \epsilon n}{ak-a+1}) + \frac{b - \epsilon n}{ak-a+1}$$ 
$$\ge g'(a-1,\, b - \frac{b - \epsilon n}{ak-a+1}) + \frac{b - \epsilon n}{ak-a+1} = g'(a, b).$$
\end{proof}

Clearly, $g(k, \,1)$ is an upper bound for $\frac{f(kn, \,n)}{n}.$ 

We have,
$$g(k, \,1) = g(k-1,\, 1- \frac{1}{k^{2} -k+1}) +\frac{1}{k^{2} -k+1}.$$ 

So Claim 1 yields that,

$$1-g(k,\, 1) = 1 - \frac{1}{k^{2}-k+1} -g(k-1, \,1)+\frac{g(k-1, \,1)}{k^{2}-k+1}.$$

Therefore we have,

$$1-g(k, \,1) = (1 - \frac{1}{k^{2}-k+1})(1-g(k-1, \,1)).$$

Now, recursively we can write,

$$1-g(k, \,1) = (1 - \frac{1}{k^{2}-k+1})(1-\frac{1}{(k-1)k-(k-1)+1})\cdots(1-\frac{1}{k-1+1}).$$

Now, if we declare $d_{a}=1-\frac{1}{ak-a+1}$ we have,

$$g(k, \,1) = 1-d_{k} d_{k-1} d_{k-2} \cdots d_{1}.$$

So we obtain: 

$$g(k,\, 1) = 1- \frac{k!}{\prod_{i=1}^{n}(\frac{k}{k-1}+i)}.$$

\end{proof}

\begin{lemm}\label{fgeg}
For any admissible constant numbers $x_{i}, k, n, m$ the following inequlity holds:

$$f(m, n) \ge g(m, n),$$

where $f$ is a good function and $g$ is a nice function both of order $k$.
\end{lemm}

\begin{proof}
If all $x_{i}=0$, then $f(m,n)=g(m,n)$ and there is nothing to prove. Suppose that for all $j \ge g(m, n)$, $x_{j} = 0$ (if they are not zero, we can change them to zero, because it will not have any effect on $g(m, n)$). Now, consider the largest $j$ for which $x_{j} > 0$. We have:
$$g(M_{j}, N_{j}) =  g(M_{j} - k (\lceil\frac{M_{j}}{N_{j}}\rceil + x_{j}), N_{j+1}) + 1 = f(M_{j} - k (\lceil\frac{M_{j}}{N_{j}}\rceil + x_{j}), N_{j+1}) + 1.$$
If we change $x_{j}$ to zero we have:
$$g(M_{j}, N_{j}) =  g(M_{j} - k (\lceil\frac{M_{j}}{N_{j}}\rceil ), N_{j+1}) + 1 =  f(M_{j} - k (\lceil\frac{M_{j}}{N_{j}}\rceil ), N_{j+1}) + 1.$$
According to Lemma \ref{monoton}, $f(M_{j} - k (\lceil\frac{M_{j}}{N_{j}}\rceil + x_{j}), N_{j+1}) \le f(M_{j} - k (\lceil\frac{M_{j}}{N_{j}}\rceil ), N_{j+1})$. Therefore by changing $x_{j}$ to zero, $g(M_{j}, N_{j})$ does not decrease. So by iteratively changing all $x_{j}$ to zero, $g(m, n)$ does not decrease and it is obvious that if all $x_{j}$ are zero, then $f(m, n) = g(m, n)$. Therefore we have $f(m, n) \ge g(m, n)$.

\end{proof}

\begin{theo}

Let $G$ be a bipartite $k$-regular graph. Then the following holds:$$\gamma_{t}(G) \le n(1- \frac{k!}{\prod_{i=0}^{k-1}(\frac{k}{k-1}+i)}).$$\label{reg_bipartite}

\end{theo}

\begin{proof}
 First suppose that $G = (X, Y)$ and $\vert X \vert = \vert Y \vert = \frac{n}{2}$. Using a greedy algorithm we will find $S \subseteq X$ such that $N(S) = Y$, and $\vert S \vert \le \frac{n}{2} - \frac {nk!}{2(\prod_{i=0}^{k-1}(\frac{k}{k - 1} + i))}$. Following the same procedure we will find $S^{\prime} \subseteq Y$ such that $N(S^{\prime}) = X$, and $\vert S^{\prime} \vert \le \frac{n}{2} - \frac {nk!}{2(\prod_{i=0}^{k-1}(\frac{k}{k - 1} + i))}$. Then $S 	\cup S^{\prime}$ is a total dominating set for $G$, of the desired size. Now, we provide a greedy algorithm to obtain $S$. First choose an arbitrary vertex $v_{1}$ of $X$ and remove all edges incident with $N(v_{1})$. Let $v_{2} \in X$ be a vertex in the new graph whose degree is maximum. Then remove all edges incident with $N(v_{2})$. Continue this procedure until all edges of $ G $ are removed. Let $S = \{v_{1},\ldots, v_{r}\}$ ($r$ is the number of selected vertices of $X$). 

We denote the number of non-isolated vertices of $X$ in the $i$th step by $n_{i}$ and the number of remaining edges by $m_{i}$. By the pigeonhole principle, in the $i$th step, there exists $v \in X$ such that $d(v) \ge \lceil \frac{m_{i}}{n_{i}}\rceil$. So suppose that $u$ has the maximum degree among all the elements of $X$ and $d(u) = \lceil \frac{m_{i}}{n_{i}}\rceil + y_{i}$ in which $y_{i}$ is a non-negative integer. Now, we define a nice function $g$ of order $k$, with $M_{0} = \frac{kn}{2}$ and $N_{0} = \frac{n}{2}$ and for all feasible $i \ge 0$, $x_{i} = y_{i+1}$. It is clear that $\vert S \vert = g(\frac{kn}{2}, \frac{n}{2})$. So the maximum possible value of $g(\frac{kn}{2},\, \frac{n}{2})$ is an upper bound for $\vert S \vert$ in our greedy algorithm.  According to Lemma \ref{fgeg} we have $g(\frac{kn}{2},\, \frac{n}{2}) \le f(\frac{kn}{2},\, \frac{n}{2})$ in which $f$ is a good function of order $k$, and finally according to Lemma \ref{mainlemma} we have $\frac{f(\frac{kn}{2},\, \frac{n}{2})}{\frac{n}{2}} \le (1- \frac{k!}{\prod_{i=0}^{k - 1}(\frac{k}{k - 1} + i)})$ which easily shows the correctness of our theorem.
\end{proof}

\begin{theo}

Suppose that $G=(X, Y)$ is a bipartite graph with |X|=|Y| and $\delta(G)\ge k$. Then the following holds:
$$\gamma_{t}(G) \le n(1- \frac{k!}{\prod_{i=0}^{k-1}(\frac{k}{k-1}+i)}).$$ \label{EqualBipartite}

\end{theo}

\begin{proof}
To prove this, we will use the same greedy algorithm as described in Theorem \ref{reg_bipartite} to choose $R \subseteq X$ such that $N(R)=Y$. Suppose that $ p =\frac{n}{2} - \frac {nk!}{2(\prod_{i=0}^{k-1}(\frac{k}{k - 1} + i))}$. By symmetry if we prove that $|R|\le p$, then similarly we can choose $R^{\prime}\subseteq Y$ such that $N(R^{\prime})=X$  and $|R^{\prime}|\le p$. Thus $R	\cup R^{\prime}$  will be a total dominating set for $G$ with the desired size. Now, we prove that $|R|\le p$ . To do this we compare the graph $G$ with $G^{\prime}=(X^{\prime}, Y^{\prime})$ which is a $k$-regular bipartite graph with $|X^{\prime}|=|X|$. Now, we apply our algorithm on $G^{\prime}$ and try to find $S \subseteq X^{\prime}$ such that $N(S)=Y^{\prime}$. We denote the number of non-isolated vertices of $X^{\prime}$ in the $i$th step of the algorithm by $n_{i}$ and the number of remaining edges of $G^{\prime}$ by $m_{i}$. It can be derived from Lemma \ref{monoton} that in each step of choosing the vertices of $S$, if the degree of the chosen vertex is as close to $\lceil\frac{m_{i}}{n_{i}}\rceil$ as possible then the number of steps of the algorithm and therefore the size of $S$ becomes as large as possible. Now, suppose that in each step $i$ of the algorithm the degree of the chosen vertex from $X^{\prime}$ is $\lceil\frac{m_{i}}{n_{i}}\rceil$. Now, if we show that $|R|\le|S|$ the proof will be complete since we have proved before that $|S| \le p$. To do this, we show that in each step of the algorithm $N(R)\ge N(S)$. By contradiction suppose that $i$ is the smallest number such that after the $(i+1)th$ step of the algorithm,  $N(R)<N(S)$.  Suppose that after the $ith$ step, $N(R)=a+b(b\ge0)$ and $N(S)=a$. Now we show that after the $(i+1)th$ step, $N(R)<N(S)$ cannot hold. To do this we must show that the following inequality cannot hold:

$$a + b + \lceil\frac{(k(n - a - b))}{(n-i)}\rceil < a + \lceil\frac{k(n-a)}{(n-i)}\rceil.$$ 

The left hand side of the inequality shows the minimum possible value of $N(R)$ after the $(i+1)th$ step and the right hand side shows the exact value of $N(S)$ after the $(i+1)th$ step.

After simplifying the inequality we obtain that $b(n-i-k)<0$. But this is obviously wrong since $b\ge0$ and also $n\ge i+k$ because after $ith$ step there still remains a vertex in $Y$ and the degree of that vertex is more or equal to $k$. So we have $n \ge i+k$. Therefore the proof is complete. 
\end{proof}

\begin{theo}
Suppose that $G=(X, Y)$ is a bipartite graph of order $n$ and $\delta(G)\ge k$. Then the following holds:
$$\gamma_{t}(G) \le n(1- \frac{k!}{\prod_{i=0}^{k-1}(\frac{k}{k-1}+i)}).$$
\end{theo}

\begin{proof}
To prove this theorem, we put two copies of graph $G$, called $A=(X_{1}, Y_{1})$ and $B=(X_{2}, Y_{2})$ together to form graph $G’=(X’, Y’)$.
We do this in a way that $X’=X_{1} \cup Y_{2}$ and $Y’=Y_{1} \cup X_{2}$. Then according to Theorem \ref{EqualBipartite}, It is evident that there is $S  \subseteq X’ \cup Y’$ such that $S$ is a total dominating set for $G’$ and $|S|<2n(1- \frac{k!}{\prod_{i=0}^{k-1}(\frac{k}{k-1}+i)})$. It is obvious that the intersection of $S$ with at least one of $A$ or $B$ is at most $n(1- \frac{k!}{\prod_{i=0}^{k-1}(\frac{k}{k-1}+i)})$. Assume that it is $A$, so $S \cap (X_{1} \cup Y_{1})$ is a total dominating set for $A$ with the required size and since $A$ is a copy of $G$, it is also a total dominating set for $G$.
\end{proof}

In the next section, we show that our result, improves Theorem \ref{MHen}.

\section{Proof of improvement over Theorem \ref{MHen}}

\begin{theo}
Let $n$ be a positive integer. Then the following holds:
$$ \frac{1+ln(n)}{n} > 1-\frac{n!}{\prod_{i=1}^{n-1} \frac{n}{n-1} +i}.$$

\end{theo}

\begin{proof}
For $n\le10$ the assertion is obvious. Now, we try to prove the inequality for $n>10$.
$$ \frac{1+ln(n)}{n} + \frac{n!}{\prod_{i=0}^{n-1} \frac{n}{n-1} +i} > 1 .$$

We have:

$$ \frac{n!}{\prod_{i=0}^{n-1} \frac{n}{n-1} +i} = \frac{1}{{\prod_{i=1}^{n} 1+\frac{1}{i(n-1)} }} .$$
Now, by considering $1 + x < e^{x}$, we rewrite the inequality as follows:

$$\frac{1+ln(n)}{n} + \frac{1}{e^{\frac{1}{n-1} \sum_{i=1}^{n} \frac{1}{i}}} > 1.$$

Since $\sum^{n}_{i=1} \frac{1}{i}<ln(n)+1$, we have:

$$\frac{1+ln(n)}{n} + \frac{1}{e^\frac{ln(n)+1}{n-1}} > 1  .$$

We declare, 

$$ f(x) = \frac{1+ln(x)}{x} + \frac{1}{e^\frac{ln(x)+1}{x-1}} - 1 .$$

Now, we rewrite the $f(x)$ as follows:

$$f(x) = \frac{1}{(xe)^\frac{1}{x-1}}+ \frac{ln(x)+1}{x}-1 .$$

It is clear that when $x$ tends to infinity, the limit of this function is $0$. So if we prove that the derivative of this function is negative for $x>10$, we can conclude that the value of this function remains positive for all $x>10$ because we know that the value of function is positive at $x=10$. 
Now, we prove that the derivative of $f(x)$ is negative for $x>10$. We have,

$$f^\prime (x)=\frac{x^{\frac{2-3x}{x-1}} (e^\frac{1}{1-x} x^2 - ((x-1)^2 x^\frac{x}{x-1} - e^\frac{1}{1-x} x^3)ln(x))}{(x-1)^2}.$$

So we have to show that the following inequality holds,

$$e^\frac{1}{1-x} x^2 - ((x-1)^2 x^\frac{x}{x-1} - e^\frac{1}{1-x} x^3)ln(x) < 0 .$$

By multipliyng both sides of the inequality by $x^{-1}e^\frac{1}{x-1}$ and considering $x^\frac{1}{x-1}=e^\frac{ln(x)}{x-1}$, we can obtain the following inequality,

$$x(xln(x) + 1) < (x-1)^2 e^\frac{ln(x)+1}{x-1} ln(x) .$$

We know that $e^ \frac{ln(x)+1}{x-1}>\frac{ln(x)+1}{x-1}+1=\frac{ln(x) + x}{x-1}$. Hence the inequality can be written as,

$$x(xln(x)+1) < (x-1)(ln(x)+x)ln(x) .$$

So the following holds,

$$ 0 < x(ln^2(x)  - ln(x) - 1) - ln^2(x) .$$
Since $ln^2(x) - ln(x) - 1$ is bigger than 1 and for $x \ge 10$, $x > ln^2(x)$, the latter inequality holds.  So the proof is complete.
\end{proof}

\bibliographystyle{amsplain} 

\end{document}